\documentclass[12pt]{amsart}
\usepackage{fullpage,amssymb,stmaryrd,url}
\usepackage{pdfx}

\newtheorem{theorem}{Theorem}[section]

\newtheorem{prop}[theorem]{Proposition}
\newtheorem{lemma}[theorem]{Lemma}
\newtheorem{cor}[theorem]{Corollary}
\theoremstyle{definition}
\newtheorem{remark}[theorem]{Remark}
\newtheorem{defn}[theorem]{Definition}
\newtheorem{hypothesis}[theorem]{Hypothesis}
\newtheorem{question}[theorem]{Question}
\newtheorem{example}[theorem]{Example}

\numberwithin{equation}{theorem}

\newcommand{\FF}{\mathbb{F}}
\newcommand{\QQ}{\mathbb{Q}}
\newcommand{\RR}{\mathbb{R}}
\newcommand{\ZZ}{\mathbb{Z}}

\newcommand{\calM}{\mathcal{M}}

\newcommand{\frakm}{\mathfrak{m}}

\DeclareMathOperator{\perf}{perf}
\DeclareMathOperator{\Spa}{Spa}
\DeclareMathOperator{\spect}{sp}

\begin{document}

\title{On commutative nonarchimedean Banach fields}
\author{Kiran S. Kedlaya}
\date{August 6, 2019}

\thanks{This work was supported by NSF grant DMS-1501214, UC San Diego
(Stefan E. Warschawski Professorship), and a Guggenheim Fellowship (fall 2015). Some of this work was carried out at MSRI during the fall 2014 research program ``New geometric methods in number theory and automorphic forms'' supported by NSF grant DMS-0932078.
Thanks to Ofer Gabber, Zonglin Jiang, and Peter Scholze for helpful discussions.}

\begin{abstract}
We study the problem of whether a commutative nonarchimedean Banach ring which is algebraically a field can be topologized by a multiplicative norm. This can fail in general, but it holds for uniform Banach rings under some mild extra conditions. Notably, any perfectoid ring whose underlying ring is a field is a perfectoid field.
\end{abstract}

\maketitle

Just as classical (commutative) Banach algebras over the real and complex numbers play a key role in analytic geometry, commutative nonarchimedean Banach algebras lie at the heart of nonarchimedean analytic geometry. When one compares algebraic geometry (in the form of the theory of schemes) to nonarchimedean analytic geometry, the role of fields in the former is best analogized in the latter by the role of fields complete with respect to multiplicative norms (commonly known as \emph{nonarchimedean fields}).

However, in certain settings, one naturally encounters a commutative nonarchimedean Banach ring whose underlying ring is a field, which for short we call a \emph{Banach field}. For example, if one starts with any commutative nonarchimedean Banach ring $R$, any maximal ideal of $I$ is closed, so the quotient $R/I$ is a Banach field. One is thus led to ask whether any Banach field is a nonarchimedean field. This fails in general 
(Example~\ref{exa:Banach field}); however, we show that this holds in some other classes of cases, such as uniform Banach algebras over fields with nondiscrete valuations (Theorem~\ref{T:no Banach field not locally compact}) and perfectoid rings
(Theorem~\ref{T:perfectoid ring field}). The latter case resolves an issue dating back to Scholze's introduction of the term \emph{perfectoid} \cite{scholze1}: 
a \emph{perfectoid field} is by definition complete with respect to a multiplicative valuation,
so Theorem~\ref{T:perfectoid ring field} is needed in order to see that this is the same thing as a perfectoid ring which is a field (or more precisely, whose underlying ring without topology is a field).

\section{Banach rings and fields}

\begin{defn}
Let $R$ be a ring. A \emph{submultiplicative (nonarchimedean) seminorm}
is a function $\left| \bullet \right|: R \to [0, \infty)$ satisfying the following conditions.
\begin{enumerate}
\item[(a)] We have $\left|0 \right| = 0$.
\item[(b)] For all $x,y \in R$, $\left| x -y \right| \leq \max\{\left| x \right|, \left| y \right|\}$.
\item[(c)] For all $x,y \in R$, $\left| xy \right| \leq \left| x \right| \left| y \right|$.
\end{enumerate}
If equality always holds in (c) and $\left| 1 \right| \neq 0$, we say that $\left| \bullet \right|$ is a \emph{multiplicative seminorm}.
A \emph{(sub)multiplicative norm}
is a (sub)multiplicative seminorm for which $\left|x \right| \neq 0$ for $x \neq 0$.
\end{defn}

\begin{remark} \label{R:operator norm}
Let $\left| \bullet \right|$ be a submultiplicative norm on a nonzero ring $R$. Define the \emph{operator norm} induced by $\left| \bullet \right|$ as the function
$\left| \bullet \right|'$ given by
\[
\left| x \right|' := \inf\{c \in \RR: \left| xy \right| \leq  c\left| y \right| \quad (y \in R)\} \qquad (x \in R).
\]
Then
\[
\left| 1 \right|^{-1} \left| x \right| \leq \left| x \right|' \leq \left| x \right| \qquad (x \in R),
\]
so $\left| \bullet \right|'$ is another submultiplicative norm  defining the same topology as $\left| \bullet \right|$ and satisfying $\left| 1 \right|' = 1$.
That is, provided that $R \neq 0$, 
there is no harm in adding the condition that $\left| 1 \right| =1$
to the definition of a submultiplicative norm. (Compare \cite[Remark~2.1.11]{part1}.)
\end{remark}

\begin{defn}
A \emph{commutative nonarchimedean Banach ring} (or for short, a \emph{Banach ring}) is a complete topological commutative ring $A$ whose topology is induced by some submultiplicative norm $\left| \bullet \right|: A  \to [0, +\infty)$.
(By Remark~\ref{R:operator norm}, if $A \neq 0$ then it is harmless to also assume that $\left| 1 \right| = 1$.)
For example, any f-adic ring in the sense of Huber which is complete and separated for its topology is a Banach ring \cite[Remark~2.4.4]{part1}. We say that $A$ is \emph{discrete} if it carries the discrete topology and \emph{nondiscrete} otherwise. (Beware that a complete topological field whose topology is induced by a discrete valuation is nondiscrete in this sense!)
\end{defn}

\begin{remark} \label{R:inverse continuous}
Let $A$ be a Banach ring
and choose a submultiplicative norm $\left| \bullet \right|$ defining the topology of $A$.
For $x \in A^\times$, let $S$ be the open neighborhood of $x$ in $A$ consisting of those $y$ for which $\left| y-x \right| < \left| x^{-1} \right|^{-1}$.
For $y \in S$, we have $\left| x^{-1} (y-x) \right| < 1$ and hence $1 + x^{-1}(y-x)$ has an inverse $z \in A$ given by summing the geometric series.
Since the map $y \mapsto z$ is evidently continuous, the inversion map from $S$ to $A$
(given by $y \mapsto x^{-1} z$) is also continuous. We deduce that $A^\times$ is open in $A$
 and the inversion map defines a homeomorphism from $A^\times$ to itself.
\end{remark}

\begin{defn}
For $A$ a Banach ring, let $A^\circ$ be the set of power-bounded elements of $A$
and let $A^{\circ \circ}$ be the set of topologically nilpotent elements of $A$.
Then $A^\circ$ is a subring of $A$, while $A^{\circ \circ}$ is an ideal of $A^\circ$ which is nontrivial if $A \neq 0$ and nonzero if $A$ is nondiscrete. Moreover, $A^{\circ \circ}$ contains $\{x \in A: \left| x \right| < 1\}$ for any submultiplicative norm $\left| \bullet \right|$ defining the topology of $A$, and hence is an open subset of $A$.
\end{defn}

\begin{defn}
By a \emph{Banach field}, we will mean a Banach ring whose underlying ring is a field. By a \emph{nonarchimedean field}, we will mean a nondiscrete Banach field whose topology is induced by some multiplicative nonarchimedean norm.
\end{defn}

The distinction between Banach fields and nonarchimedean fields is important in part because of the following basic fact.
\begin{lemma} \label{L:field completion}
Let $F$ be a field equipped with a nontrivial multiplicative nonarchimedean norm. Then the completion of $F$ is a field (and hence a nonarchimedean field).
\end{lemma}
\begin{proof}
Let $\widehat{F}$ be the completion of $F$ and let $x \in \widehat{F}$ be a nonzero element.
Choose a Cauchy sequence $x_1,x_2,\dots$ in $F$ converging to $x$.
Since this sequence does not converge to zero, the sequence $\left| x_1 \right|, \left| x_2 \right|, \dots$ does not converge to 0 either; by passing to a subsequence, we may ensure that this sequence is bounded away from 0. Since the norm is multiplicative, it follows that the sequence $\left| x_1^{-1} \right|, \left| x_2^{-1} \right|, \dots$ is bounded. By writing
\[
x_n^{-1} - x_{n+1}^{-1} = (x_{n+1}-x_n) x_n^{-1} x_{n+1}^{-1},
\]
we see that $x_1^{-1},x_2^{-1},\dots$ is also a Cauchy sequence, so it has a limit $y \in \widehat{F}$; we must then have $xy= 1$. Hence $\widehat{F}$ is a field.
\end{proof}
\begin{remark}
By contrast, completing a field with respect to a submultiplicative norm generally does not yield a field. For example, the completion of $\QQ$ with respect to the supremum of the 2-adic and 3-adic absolute values yields the direct sum $\QQ_2 \oplus \QQ_3$.
For a generalization of this observation, see Remark~\ref{R:connected}.
For a more exotic example, see Example~\ref{exa:not Banach field}.
\end{remark}

\begin{remark}
An important, but not presently relevant, theorem of Schmidt (e.g., see \cite[Theorem~4.4.1]{engler-prestel}) asserts that a field which is not separably closed can be topologized as a nonarchimedean field in at most one way.
\end{remark}

\section{Banach algebras and their spectra}

\begin{defn} \label{D:uniform}
Let $A$ be a Banach ring with a specified submultiplicative norm $\left| \bullet \right|$.
The \emph{Gel'fand spectrum} of $A$ is the set $\calM(A)$ of multiplicative seminorms on $A$ bounded above by the specified norm (note that the zero function is excluded). It has been shown by Berkovich
\cite[Theorem~1.2.1]{berkovich1} (see also \cite[\S 2.3]{part1}) that $\calM(A) \neq \emptyset$ if $A \neq 0$, and moreover the \emph{spectral seminorm}
\[
\left| x \right|_{\spect} := \lim_{n \to \infty} \left| x^n \right|^{1/n}
\]
satisfies
\begin{equation} \label{eq:spectrum}
\left| x\right|_{\spect} = \max\{\alpha: \alpha(x) \in \calM(A)\} \qquad (x \in A);
\end{equation}
in particular, the maximum is achieved. For $\alpha \in \calM(A)$, define
$\ker(\alpha) := \alpha^{-1}(0)$; this is a prime ideal of $A$.

We say that $A$ is \emph{uniform} if $A^{\circ}$ is bounded in $A$. If $A$ is Tate, it is equivalent to require that $\left| x \right|_{\spect}$ is a norm defining the same topology as the originally specified norm; compare Remark~\ref{R:single point1} and
\cite[\S 2.8]{part1}. In general, the separated completion of $A$ with respect to the spectral seminorm is another Banach ring, denoted by $A^u$ and called the \emph{uniformization} of $A$; note that the natural map $A \to A^u$ induces a bijection $\calM(A^u) \cong \calM(A)$ (which is a homeomorphism for the topology described in Remark~\ref{R:connected} below).
\end{defn}

\begin{remark} \label{R:underlying topology}
Note that while the definition of $\calM(A)$ refers to a norm,
the underlying set of $\calM(A)$ depends only on the underlying topology of $A$.
This may be seen by identifying the elements of $\calM(A)$ with nonzero continuous homomorphisms from $A$ to nonarchimedean fields, as in \cite[Remarks~1.2.2(ii)]{berkovich1}, or with equivalence classes of continuous real semivaluations on $A$,
as in Remark~\ref{R:single point2} below.
\end{remark}

\begin{remark} \label{R:connected}
The set $\calM(A)$ may be viewed as a compact topological space via the evaluation topology (see \cite[\S 2.3]{part1}); as in Remark~\ref{R:underlying topology}, this structure only depends on the underlying topology of $A$ and not a choice of norm.
If $\calM(A)$ is disconnected for the evaluation topology, then by \cite[Proposition~2.6.4]{part1} $A$ is also disconnected (that is, it contains a nontrivial idempotent). In particular, if $A$ is a Banach field, then $\calM(A)$ is connected.
\end{remark}

We will see shortly that the spectrum of a nonarchimedean field is a single point (Remark~\ref{R:single point2}). That statement has the following partial converse; note that the conclusion must accommodate examples like $A = F[T]/(T^2)$ for some nonarchimedean field $F$, for which $\calM(A) \cong \calM(F)$.

\begin{lemma} \label{L:point to field}
If $\calM(A)$ consists of a single point $\alpha$, then $A/\ker(\alpha)$ is a Banach field
and $A^u$ is a nonarchimedean field.
\end{lemma}
\begin{proof}
It is an easy consequence of \cite[Theorem~1.2.1]{berkovich1}
that $f \in A$ is invertible if and only if $\alpha(f) \neq 0$ for all $\alpha \in \calM(A)$
(see \cite[Corollary~1.2.4]{berkovich1} for the full argument). This immediately implies the first assertion. The second assertion follows from the first assertion plus \eqref{eq:spectrum}.
\end{proof}

\begin{defn}
A \emph{Banach algebra} over a nonarchimedean field $F$ is a Banach ring $A$ equipped 
with a continuous homomorphism $F \to A$. 
\end{defn}

\begin{lemma} \label{L:top nil unit}
Let $A$ be a nondiscrete Banach field. Assume that either:
\begin{itemize}
\item[(i)]
$A$ is of characteristic $p$;
\item[(ii)]
$A$ is of characteristic $0$ and its topology is induced by some submultiplicative norm under which $\QQ$ is bounded; or
\item[(iii)]
$A$ is of characteristic $0$, the induced topology on $\QQ$ is not discrete, and case (ii) does not apply.
\end{itemize}
Then $A$ is a Banach algebra over some nondiscrete nonarchimedean field.
\end{lemma}
\begin{proof}
In case (i), since $A^{\circ \circ}$ is open in $A$ and $A$ is nondiscrete, we can find some nonzero $z \in A^{\circ \circ}$. We then obtain a continuous map $\FF_p((z)) \to A$. Case (ii) is similar.

In case (iii), we have $\ZZ \subseteq A^\circ$. 
Since $A^{\circ \circ}$ is open in $A$, the intersection $I = \ZZ \cap A^{\circ \circ}$ is a nonzero ideal of $\ZZ$, which is thus generated by some positive integer $n$. We cannot have $n=1$ since $1 \notin A^{\circ \circ}$. On the other hand, the completion of $\ZZ$ with respect to $I$ factors as the product of $\ZZ_p$ over all prime factors $p$ of $n$; since $A$ is an integral domain, this forces $n$ to be a power of a single prime $p$, and thus to equal $p$ itself
(since $A^{\circ \circ}$ is a radical ideal). We may thus view $A$ as a Banach algebra over the nonarchimedean field $\QQ_p$.
\end{proof}

\begin{remark}
A typical example of a submultiplicative norm on $\QQ$ which is unbounded but induces the discrete topology is given by setting $\left| x \right| := e^{n(x)}$ for $x \neq 0$,
where $n = n(x)$ is the smallest nonnegative integer for which $(n!)^n x \in \ZZ$. 
In Example~\ref{exa:Q series} we will see an example of a nondiscrete Banach field whose topology is induced by a submultiplicative norm restricting to this norm on $\QQ$; this will show that the hypotheses of Lemma~\ref{L:top nil unit} are necessary.
\end{remark}

In light of Lemma~\ref{L:top nil unit}, we focus most of our attention on Banach algebras over nonarchimedean fields.

\begin{hypothesis} \label{H:Banach algebra1}
Hereafter, 
fix a nonarchimedean field $F$ and a multiplicative norm $\left| \bullet \right|_F$ on $F$.
Let $\kappa_F$ denote the residue field of $F$.
Let $\left| F^\times \right| \subseteq \RR^+$ denote the norm group of $F$.
\end{hypothesis}

\begin{lemma} \label{L:compatible norm}
Let $A$ be a Banach algebra over $F$. Then
the topology of $A$ is induced by some submultiplicative norm $\left| \bullet \right|_A$ satisfying $\left| xy \right|_A = \left| x \right|_F \left| y \right|_A$ for all $x \in F, y \in A$. In particular, if $A \neq 0$, then we may further ensure that $\left|1 \right|_A = 1$, 
and then 
$\left| x \right|_F = \left| x \right|_A$ for all $x \in F$.
\end{lemma}
\begin{proof}
Start with any submultiplicative norm $\left| \bullet \right|$ defining the topology of $A$
and let $M$ denote the unit ball in $A$ for this norm.
Let $\mathfrak{o}_F$ be the valuation ring of $F$ and define the function $\left| \bullet \right|_A$ by the formula
\[
\left| x \right|_A = \inf\{\left| y \right|: y \in F, x \in y \mathfrak{o}_F M\}.
\]
This then has the desired effect (see \cite[\S 1.2]{schneider}).
\end{proof}

\begin{remark} \label{R:single point1}
Suppose that $A$ is both a nonarchimedean field and a Banach algebra over $F$. 
Then the topology of $A$ is defined both by a multiplicative norm $\left| \bullet \right|_1$ and by a submultiplicative norm $\left| \bullet \right|_A$
as in Lemma~\ref{L:compatible norm}. One way this can occur is if 
$\left| \bullet \right|_1$ and $\left| \bullet\right|_A$ are \emph{norm-equivalent} in the sense that there exist constants $c_1, c_2 > 0$ such that
\begin{equation} \label{eq:norm equivalence}
\left| x \right|_1 \leq c_1 \left| x \right|_A, \qquad
\left| x \right|_A \leq c_2 \left| x \right|_1;
\end{equation}
however, this need not be the case for the given norms. That said, the restriction of $\left| \bullet \right|_1$ to $F$ is a multiplicative norm defining the same topology as the multiplicative norm $\left| \bullet \right|_F$, so there must exist a single constant $c>0$ such that
$\left| x \right|_1 = \left| x \right|_A^c$ for all $x \in F$. By replacing the original norm $\left| \bullet \right|_1$ on $A$ with the new multiplicative norm $\left| \bullet \right|_1^{1/c}$, we may arrive at the situation where $c=1$, in which case $\left| \bullet \right|_1$ and $\left| \bullet\right|_A$ are norm-equivalent in the sense
of \eqref{eq:norm equivalence}.
In particular, $\left| \bullet \right|_{\spect} = \left| \bullet \right|_1$ is multiplicative,
and moreover is uniquely determined by the topologies on $A$ and $F$ and the norm on $F$.
\end{remark}

\begin{remark} \label{R:single point2}
In Remark~\ref{R:single point1}, take $A = F$ and suppose that $\alpha \in \calM(A)$. From \eqref{eq:spectrum}, we see that $\alpha(x) \leq \left| x \right|_1$ for all $x \in A$.
For $x \in A^\times$, the same holds with $x$ replaced by $x^{-1}$;
since $\alpha$ and $\left|  \bullet \right|_1$ are multiplicative, 
this yields
\[
\alpha(x) \leq \left| x \right|_1 = \left| x^{-1} \right|_1^{-1}
\leq \alpha(x^{-1})^{-1} = \alpha(x).
\]
We deduce that $\calM(A)$ consists of the single point $\alpha = \left| \bullet \right|_1$, independent of the choice of the norm $\left| \bullet \right|_A$.
(Compare \cite[Corollary~1.3.4]{berkovich1}.)
\end{remark}

\begin{remark} \label{R:nondiscrete field construction}
Let $A$ be an affinoid algebra over $F$ in the classical (Tate) sense \cite[\S 6.1]{bgr}. By the Nullstellensatz for affinoid algebras \cite[Corollary~6.1.2/3]{bgr}, every maximal ideal of $A$ has a residue field which is finite over $F$, and hence a nonarchimedean field \cite[Theorem~3.2.1/2]{bgr}. In particular, if $A$ is a Banach field, then $A$ is a nonarchimedean field.

Suppose now that $A$ is an affinoid algebra over $F$ in the more general sense of Berkovich \cite[Definition~2.1.1]{berkovich1}. Then it is no longer the case that every maximal ideal of $A$ has residue field finite over $F$. For example, suppose that $\rho>0$ is not in the divisible closure of $\left| F^\times \right|$. Form the completion $F\{T/\rho, U/\rho^{-1}\}$ of $F[T,U]$ for the weighted Gauss norm with weights $\rho, \rho^{-1}$, then let $A$ be the quotient of this ring by the ideal $(TU-1)$; then $A$ is itself a nonarchimedean field. Consequently, the method of the previous paragraph does not suffice to show that an affinoid algebra which is a Banach field is a nonarchimedean field; however, this does turn out to be true by another argument (see Proposition~\ref{P:affinoid}).
\end{remark}

The following lemma and proof were suggested by Gabber.
\begin{lemma} \label{L:circles}
Let $A$ be a connected affinoid algebra over $F$ in the sense of Berkovich. 
Fix a homomorphism $f: F\{T\} \to A$ such that the image of $T$ in $A$ is invertible.
Let $\rho, \sigma$ be the spectral norms of $T, T^{-1}$ in $A$.
Then there exists a finite set $S \subset \RR$ such that the image of the map $f^*: \calM(A) \to \calM(F\{T\})$
includes all points $\alpha$ with $\alpha(T) \in [\sigma^{-1}, \rho] \setminus S$.
\end{lemma}
\begin{proof}
Suppose first that $F$ is algebraically closed
and that $A$ is an affinoid algebra in the sense of Tate. 
Since quotienting $A$ by its nilradical does not change $\calM(A)$, we may assume that $A$ is reduced.
Since $A$ is noetherian \cite[Proposition~2.1.3]{berkovich1}, $A$ has finitely many minimal prime ideals;
by applying the following argument to the quotients by these ideals, we may reduce to the case where $A$ is an integral domain. There is nothing to check if $A$ is a finite extension of $F$ (as then $\sigma \rho = 1$).
Otherwise, since $F\{T\}$ is a principal ideal domain, $F\{T\} \to A$ is flat
and we may use the Bosch-L\"utkebohmert flattening method \cite[Corollary~5.11]{bosch-lutkebohmert}
to see that the image $U$ of $\calM(A)$ in $\calM(F\{T\})$ is a (connected) finite union of affinoid subdomains. Since $F$ is algebraically closed,
every connected affinoid subdomain of $\calM(F\{T\})$ consists of some closed disc minus a finite union of open discs. Since $\calM(A)$ is connected by Remark~\ref{R:connected}, the map $U \to [\sigma^{-1}, \rho]$ taking $\alpha$ to $\alpha(T)$ must be surjective; consequently, $U$ must be the complement in the annulus $\sigma^{-1} \leq \left| T \right| \leq \rho$ of a finite union of open discs, each of which is contained in the circle $\left| T \right| = \tau$ for some $\tau \in [\sigma^{-1}, \rho]$. This proves the claim.

We now treat the general case. Let $F'$ be an algebraically closed
nonarchimedean field containing $F$ such that $A_{F'} := A \widehat{\otimes}_F F'$ is an affinoid algebra in the sense of Tate. Since $\calM(A_{F'})$ has finitely many connected components and  $\calM(A_{F'}) \to \calM(A)$ is surjective, we may apply the previous paragraph to each connected component of $A_{F'}$ to conclude.
\end{proof}

\begin{prop} \label{P:affinoid}
Let $A$ be an affinoid algebra over $F$ in the sense of Berkovich which is a Banach field. Then $A$ is a nonarchimedean field.
\end{prop}
\begin{proof}
Since $A$ is reduced, it is uniform \cite[Proposition~2.1.4]{berkovich1};
we may thus equip $A$ with its spectral norm.
Suppose that there exists some nonzero $T \in A$ such that the quantities $\rho := \left| T \right|$,
$\sigma := \left| T^{-1} \right|$ do not satisfy $\rho \sigma = 1$. We may then apply Lemma~\ref{L:circles} to find some quantity $\tau \in [\sigma^{-1}, \rho]$ in the divisible closure of $\left| F^\times \right|$ such that $\calM(A)$ covers the circle $\left| T \right| = \rho$ in the Berkovich analytic affine $T$-line over $F$.
Inside this circle, we may choose a point $\alpha$ corresponding to the rigid-analytic point at which some irreducible polynomial $P \in F[T]$ vanishes. But now for any lift $\beta \in \calM(A)$ of $\alpha$, we must have $\beta(P(T)) = 0$ and so $P(T)$ cannot be a unit in $A$. We must therefore have $P(T) = 0$ in $A$; however, in this case the image of $F[T]$ in $A$ would be a finite extension of $F$ and hence a nonarchimedean field, and so we would have $\rho \sigma = 1$, a contradiction.

From the previous paragraph, we see that $\left| T \right| \left| T^{-1} \right|^{-1} = 1$ for all nonzero $T \in A$. This implies that $\calM(A)$ consists of a single point, so we may apply Lemma~\ref{L:point to field}
to deduce that $A$ is a nonarchimedean field.
\end{proof}

For a slightly more exotic example, consider the following example of a Banach field which is not a nonarchimedean field, but whose spectrum is again reduced to a single point.
\begin{example} \label{exa:Banach field}
Let $\alpha$ be the Gauss norm on the rational function field $F(T_1,T_2, \dots)$ in countably many variables. 
Define the function $f: F(T_1,T_2,\dots) \to \ZZ$ taking $x$ to the smallest nonnegative integer $k$ such that $x \in F(T_1,\dots,T_k)$.
Let $A$ be the completion of $F(T_1, T_2, \dots)$ for the norm 
\[
\left| x \right| = \inf\{ \max\{2^{f(x_i)} \alpha(x_i): i=1,\dots,n\}: x = x_1 + \cdots + x_n \}.
\]
For $x \in F(T_1,T_2,\dots)$, we have $\left| x \right|_{\spect} = \alpha(x)$, so the restriction of $\left| \bullet \right|_{\spect}$ to $F(T_1,T_2,\dots)$ equals the multiplicative norm $\alpha$.
Hence $A^u$ equals the completion of $F(T_1,T_2,\dots)$ with respect to $\alpha$,
which by Lemma~\ref{L:field completion} is a nonarchimedean field.

Let $A_k$ be the completion of $F(T_1,\dots,T_k)$ with respect to $\alpha$,
or equivalently with respect to $\left| \bullet \right|$.
Let $B_k$ be the ring of formal sums $\sum_{n \in \ZZ} a_n T_{k+1}^n$ with $a_n \in A_k$
such that $\alpha(a_n)$ remains bounded as $n \to \infty$ and tends to 0 as $n \to -\infty$; this ring is complete for the (multiplicative) Gauss norm. Identify $F[T_1,\dots,T_{k+1}]$ with $F[T_1,\dots,T_k][T_{k+1}]$, then map the latter to $A_k[T_{k+1}]$ and on to $B_k$; the resulting map carries every nonzero element of $F[T_1,\dots,T_{k+1}]$ to a unit in $B_k$. We thus obtain an isometric ring embedding $A_{k+1} \to B_k$. The composition $A_k \to A_{k+1} \to B_k$ is split by the projection map $B_k \to A_k$ of $A_k$-modules taking $\sum_{n \in \ZZ} a_n T_{k+1}^n$ to $a_0$.
The compositions $A_{k+1} \to B_k \to A_k$ are submetric with respect to both $\alpha$ and $\left| \bullet \right|$; by chaining these together, we get a compatible family of submetric projections $A_n \to A_k$ for all $n \geq k$, and by continuity also a projection $\pi_k: A \to A_k$.

For any $x \in F(T_1,T_2,\dots)$, we have
\begin{equation} \label{eq:monotone}
\left| \pi_1(x) \right| \leq \left| \pi_2(x) \right| \leq \cdots, \qquad
\left| \pi_1(x) \right|_{\spect} \leq \left| \pi_2(x) \right|_{\spect} \leq \cdots,
\end{equation}
and the sequence $\pi_1(x), \pi_2(x), \dots$ stabilizes at the constant value $x$.
By continuity, it follows that for any $x \in A$, \eqref{eq:monotone} holds and the sequence $\pi_1(x), \pi_2(x),\dots$ is a Cauchy sequence with limit $x$.
In particular, if $x \in A$ satisfies $\left| x \right|_{\spect} = 0$, then $\pi_k(x) = 0$ for all $k$,
so $x = 0$; that is, the map $A \to A^u$ is injective.
By Remark~\ref{R:single point2}, $\calM(A) = \calM(A^u)$ consists of the single point $\alpha$;
by Lemma~\ref{L:point to field}, we deduce that $A$ is a Banach field.

For each $k$, we have $\left| T_k \right| = 2^k$ while $\left| T_k \right|_{\spect} = 1$; hence $\left| \bullet \right|_{\spect}$ does not define the topology of $A$, so $A$ is not uniform. In particular, $A$ is not a nonarchimedean field.
\end{example}

Another example of a Banach field which is not uniform can be found in \cite[\S 8.3]{fujiwara-gabber-kato}.
We include a modification of this example suggested by Gabber, to show that the hypotheses in Lemma~\ref{L:top nil unit} cannot be weakened.
\begin{example} \label{exa:Q series}
Form the ring 
\[
A_0 := \left\{ \sum_{n=0}^\infty a_n T^n: a_n \in \frac{1}{(n!)^n} \ZZ \right\} \subseteq \QQ [T],
\]
let $A_1$ be the $T$-adic completion of $A_0$, and put $A := A_1[T^{-1}]$. For $x \in A$, let $\left| x \right|_A$ be the infimum of $e^{-n}$ over all $n \in \ZZ$ for which $T^{-n} x \in A_1$; with the topology induced by $\left| \bullet \right|$, $A$ is a Banach field. The unit ball for the spectral seminorm equals $A \cap \QQ \llbracket T \rrbracket$, which is not bounded under $\left| \bullet \right|_A$; hence $A$ is not uniform, and in particular not itself a nonarchimedean field.

Note that $\QQ \cdot A_0$ is not a bounded subset of $A$; consequently, the topology on $A$ cannot be defined
by any submultiplicative norm under which $\QQ$ is bounded. By Lemma~\ref{L:compatible norm}, it follows that $A$ is not a Banach algebra over any nonarchimedean field.
\end{example}

\begin{remark}
The Banach field $A$ constructed in \cite[\S 8.3]{fujiwara-gabber-kato} has the additional feature that $A \{T\}$ is not noetherian. In particular, $A$ cannot be a nonarchimedean field in light of the Hilbert basis theorem for Tate algebras over nonarchimedean fields \cite[Theorem~5.2.6/1]{bgr}; it also provides an explicit example of the failure of the general Hilbert basis theorem for commutative nonarchimedean Banach rings.
By contrast, we do not know whether or not $A\{T\}$ is noetherian in the case where $A$ is the field described in Example~\ref{exa:Banach field}.
\end{remark}

We leave the following question completely unaddressed.
\begin{question}
Does there exist an example of a Banach field whose uniform completion is not a nonarchimedean field? (Note that we do not require the uniform completion to itself be a Banach field.) Gabber points out that Escassut \cite{escassut} has conjectured the negative answer for Banach fields over a nonarchimedean field.
\end{question}

\section{Uniform Banach fields}

In light of the key role played in Example~\ref{exa:Banach field} by the failure of uniformity,
we consider the following question.

\begin{question} \label{Q:Banach field}
Is every uniform Banach field a nonarchimedean field?
\end{question}

We will only treat this question for Banach algebras over a nonarchimedean field, so let us immediately restrict to this case.

\begin{hypothesis} \label{H:Banach algebra}
Hereafter, let $A$ be a uniform Banach algebra over $F$, equipped with a norm given by Lemma~\ref{L:compatible norm}. Note that by Remark~\ref{R:single point1} the associated spectral seminorm is independent of any choices (except for the initial choice of the norm on $F$). Moreover, since $A$ is uniform, the spectral seminorm itself is a norm defining the topology of $A$. 
\end{hypothesis}

\begin{remark} \label{remark:field from spectrum}
By Lemma~\ref{L:point to field} and Remark~\ref{R:single point2},
$A$ is a nonarchimedean field if and only if $\calM(A)$ is a single point.
In particular, if $A$ is a uniform Banach field, then $A$ is a nonarchimedean field if and only if for every $t \in A^\times$, we have $\left|t \right|_{\spect} \left| t^{-1} \right|_{\spect} = 1$. Note that this condition may be checked within the completion of $F(t)$ inside $A$.
\end{remark}

With Remark~\ref{R:connected} and Remark~\ref{remark:field from spectrum}
in mind, one is naturally led to try to exhibit a negative answer
to Question~\ref{Q:Banach field} by completing $F(t)$ with respect to a connected set of norms. However, a straightforward attempt of this sort fails in an instructive way.
\begin{example} \label{exa:not Banach field}
Suppose that $\left |F^\times \right|$ is dense in $\RR^+$ (i.e., $F$ is not discretely valued).
Choose a closed interval $I = [\gamma, \delta] \subset (0, +\infty)$ of positive length.
Let $A$ be the completion of $F(t)$ with respect to the supremum of the $\rho$-Gauss norms
$\left| \bullet \right|_\rho$ for all $\rho \in I$. By construction, $A$ is a uniform Banach ring, and it can (but need not here) be shown that $\calM(A)$ is homeomorphic to $I$ via the map taking the $\rho$-Gauss norm to $\rho$. However, despite the fact that $A$ is the completion of the field $F(t)$, we will show that $A$ is not a Banach field.

Choose a strictly increasing sequence
$\rho_1, \rho_2, \dots$ in $\left| F^\times \right| \cap I$ with limit $\delta$,
then choose a sequence $m_1, m_2, \dots$ of positive integers such that
\[
\rho_{n-1}/\rho_n,
\rho_n/\rho_{n+1}
 \leq 2^{-n/m_n} \qquad (n>1).
\]
For $n \geq 1$, choose elements $\lambda_{n,1}, \dots, \lambda_{n,2m_n-1} \in F$ of norm $\rho_n$, write
\[
P_n(t) := \prod_{i=1}^{2m_n-1} (t - \lambda_{n,i}) = \sum_{j=0}^{2m_n-1} P_{n,j} t^j,
\]
and consider the following sequences $x_1,x_2,\dots$ and $y_1,y_2,\dots$ in $A$:
\[
x_n := \frac{\sum_{j=0}^{m_n-1} P_{n,j} t^j}{P_n(t)}, \qquad y_n := x_1 \cdots x_n.
\]
We make the following observations.
\begin{itemize}
\item
If $\rho < \rho_n$, then $\left| x_n \right|_\rho = 1$,
$\left| 1-x_n \right|_\rho \leq (\rho/\rho_n)^{m_n} < 1$.
\item
If $\rho > \rho_{n}$, then $\left| 1-x_n \right|_\rho = 1$,
$\left| x_n \right|_\rho \leq (\rho_n/\rho)^{m_n} < 1$. In particular, $\left| x_n \right|_\delta \leq (\rho_n/\delta)^{m_n} < (\rho_n/\rho_{n+1})^{m_n} \leq 2^{-n}$.
\item
If $\rho = \rho_n$, then $\left| x_n \right|_\rho = \left| 1-x_n \right|_\rho = 1$. Hence for all $n$,
\[
\left| x_n \right|_{\spect}, \left| 1-x_n \right|_{\spect}, \left| y_n \right|_{\spect} \leq 1.
\]
\item
For $n>1$ and $\rho \in [\gamma, \rho_{n-1}]$, 
we may write $y_n - y_{n-1} = y_{n-1}(x_n - 1)$ to obtain
\[
\left| y_n - y_{n-1} \right|_\rho \leq \left| 1 - x_n \right|_\rho \leq (\rho/\rho_n)^{m_n}
\leq (\rho_{n-1}/\rho_n)^{m_n} \leq2^{-n}.
\]
\item
For $n>3$ and $\rho \in [\rho_{n-1}, \delta]$,
we may write $y_n - y_{n-1} = y_{n-3} x_{n-1} (x_n - 1) x_{n-2}$ to obtain 
\[
\left| y_n - y_{n-1} \right|_\rho \leq \left| x_{n-2} \right|_\rho 
\leq (\rho_{n-2}/\rho)^{m_{n-2}} \leq (\rho_{n-2}/\rho_{n-1})^{m_{n-2}} \leq 2^{-n+2}.
\]
\end{itemize}
We now see that for all $n > 3$,
$\left| y_n - y_{n-1} \right|_{\spect}  \leq 2^{-n+2}$.
In particular, the sequence $y_1, y_2,\dots$ is Cauchy and so has a limit $y \in A$.
By construction, we have $\left| y_n \right|_{\delta} \to 0$ as $n \to \infty$, so
$\left| y \right|_{\delta} = 0$. On the other hand, for $\rho \in [\gamma, \delta)$, 
the sequence $\left| y_n \right|_\rho$ is eventually constant, so $\left| y \right|_\rho > 0$ and in particular $y \neq 0$. In particular, $y$ is neither zero nor a unit in $A$, and hence $A$ is not a Banach field.
\end{example}

\begin{remark} \label{R:not Banach field}
The proof of Theorem~\ref{T:no Banach field not locally compact} has its origins in the observation that the construction of Example~\ref{exa:not Banach field} is quite robust.
For example, suppose that $F$ has infinite residue field, and replace $A$ by the completion of $F(t)$ with respect to the supremum of the 
$\rho$-Gauss norms $\left| \bullet \right|_\rho$ for all $\rho \in I$
plus the $(\mu \rho)$-Gauss norms on $F(t-\lambda)$ for all $\rho \in I, \lambda \in F$ with $\left| \lambda \right| = \rho$ and all $\mu \in [\frac{1}{2}, 1]$.
In the construction of $x_n$, add the restriction that for each $n$, the ratios
$\lambda_{n,i}/\lambda_{n,1}$ for $i=1,\dots,2m_n-1$ represent distinct elements of the residue field of $F$.
Then 
\[
\left| x_n \right|_\rho \begin{cases} = (\rho/\rho_n)^{m_n} & (\rho < \rho_n) \\
\leq 2 & (\rho = \rho_n) \\
= (\rho_n/\rho)^{m_n} & (\rho > \rho_n), \end{cases}
\]
from which it follows that again the sequence $y_1,y_2,\dots$ converges to an element $y \in A$ which is neither zero nor a unit. As the details are quite similar to the proof of Theorem~\ref{T:no Banach field not locally compact}, we omit them here.
\end{remark}

By carrying this reasoning further, we obtain a substantial partial answer to Question~\ref{Q:Banach field}.

\begin{lemma} \label{L:interval}
Suppose that $A$ (which by Hypothesis~\ref{H:Banach algebra} is a uniform Banach algebra over $F$)
is a Banach field but not a nonarchimedean field. 
Choose any $c>1$ such that $(1, c] \cap \left| F^\times \right| \neq \emptyset$.
Then there exist:
\begin{itemize}
\item
a nonzero element $t \in A$ with $\left| t \right|_{\spect} \left| t^{-1} \right|_{\spect} > 1$;
\item
values $\gamma, \delta$ with
$\left| t^{-1} \right|_{\spect}^{-1} \leq \gamma < \delta \leq \left| t \right|_{\spect}$
and $\left| F^\times \right| \cap [\gamma, \delta] \neq \emptyset$;
\end{itemize}
such that for every $\lambda \in F$ with $\left|\lambda \right| \in [\gamma, \delta]$,
$\left| \lambda \right| \left|( t - \lambda)^{-1} \right|_{\spect} \leq c^{2}$.
\end{lemma}
\begin{proof}
We will assume the contrary and derive a contradiction. 
Since $A$ is not a nonarchimedean field,
there must exist some nonzero $t_0 \in A$ for which $\left| t_0 \right|_{\spect} \left| t_0^{-1} \right|_{\spect} \neq 1$; by replacing $t_0$ with a suitable power, we may further ensure that $\left| t_0 \right|_{\spect} \left| t_0^{-1} \right|_{\spect} > c$. 
Put $\gamma_0 := \left| t_0^{-1} \right|_{\spect}^{-1}$, $\delta_0 := c \gamma_0$.

For $n=0,1,\dots$, we construct elements $\mu_n \in F$ 
and real numbers $\gamma_n, \delta_n$ with $\gamma_{n+1} < \gamma_n/c$ 
such that for $t_n = t_0 - \mu_n$,
we have
\[
\left| t_n^{-1} \right|_{\spect}^{-1} \leq \gamma_n,
\qquad c \gamma_n = \delta_n, \qquad \delta_n < \left| t_n \right|_{\spect}.
\]
To begin with, put $\mu_0 = 0$ and consider $\gamma_0, \delta_0$ as above.
Given $\mu_n,\gamma_n, \delta_n$ for some $n$,
note that by hypothesis, the conditions of the lemma do not hold for $t = t_n$, $\gamma = \gamma_n$, $\delta = \delta_n$;
that is, there exists $\lambda_n \in F$ with $\left| \lambda_n \right| \in [\gamma_n, \delta_n]$ such that
$\left| \lambda_n \right| \left|(t_n - \lambda_n)^{-1} \right|_{\spect} > c^{2}$.
Put 
\[
\mu_{n+1} := \mu_n + \lambda_n, \qquad
\gamma_{n+1} := \left| t_{n+1}^{-1} \right|^{-1}_{\spect},
\qquad
\delta_{n+1} := c \gamma_{n+1}.
\]
From the construction, we have $\gamma_{n+1} < \delta_n/c^2 = \gamma_n/c$, so 
$\delta_{n+1} \leq \gamma_n < \delta_n$. Since $\left| t_{n+1} \right|_{\spect} = 
\left| t_n \right|_{\spect}$, the elements $\mu_{n+1}, \gamma_{n+1}, \delta_{n+1}$ have the desired properties.

Since $\left| \lambda_n \right| \leq \delta_n$ and both $\gamma_n$ and $\delta_n$ tend to 0 as $n \to \infty$, the sequence $\{\mu_n\}_{n=0}^\infty$ converges to a limit $\mu \in F$.
We cannot have $t_0 = \mu$, as this would imply $t_0 \in F$ and hence
$\left| t_0 \right|_{\spect} \left| t_0^{-1} \right|_{\spect} = 1$.
Consequently, $t_0 - \mu$ is a unit in $A$, and so $\{(t_0 - \mu_n)^{-1}\}_{n=0}^\infty$ converges to $(t_0 - \mu)^{-1}$.
However, the sequence $\left| (t_0 - \mu_n)^{-1} \right|_{\spect}$ does not converge to
$\left| (t_0 - \mu)^{-1} \right|_{\spect}$; instead, it diverges to $+\infty$. This contradiction yields the desired result.
\end{proof}

\begin{theorem} \label{T:no Banach field not locally compact}
Suppose that $\left| F^\times \right|$ is dense in $\RR^+$.
If $A$ is a uniform Banach field, then $A$ is a nonarchimedean field.
\end{theorem}
\begin{proof}
Suppose by way of contradiction that $A$ is a uniform Banach field but not a nonarchimedean field.
Set notation as in Lemma~\ref{L:interval}; we will use this framework to carry out a variant of Example~\ref{exa:not Banach field} in the manner of Remark~\ref{R:not Banach field}.
Choose a strictly increasing sequence $\rho_1, \rho_2, \dots$ in $[\gamma, \delta]$ with limit $\delta$.
Choose a nondecreasing sequence $m_0, m_1, \dots$ of positive integers such that
\[
\rho_{n-1}/\rho_n,
\rho_n/\rho_{n+1}
 \leq 2^{-n/m_{n-1}} \qquad (n>1).
\]
Choose $\rho_1^-, \rho_1^+, \rho_2^-, \rho_2^+ \dots$ with 
$\rho_1^- < \rho_1 < \rho_1^+ < \rho_2^- < \rho_2 < \cdots$ and
\[
\rho_n/\rho_n^-, \rho_n^+/\rho_n \leq 2^{1/m_n}.
\]
For $n \geq 1$, choose elements $\lambda_{n,1}, \dots, \lambda_{n,2m_n-1} \in F$ whose norms are pairwise distinct elements of $[\rho_n^-, \rho_n^+]$.
Write
\[
P_n(t) := \prod_{i=1}^{2m_n-1} (t - \lambda_{n,i}) = \sum_{j=0}^{2m_n-1} P_{n,j} t^j,
\]
and consider the following sequences $x_1,x_2,\dots$ and $y_1,y_2,\dots$ in $A$:
\[
x_n := \frac{\sum_{j=0}^{m_n-1} P_{n,j} t^j}{P_n(t)}, \qquad y_n := x_1 \cdots x_n.
\]
For $\alpha \in \calM(A)$ with $\alpha(t) = \rho$, we make the following observations. 
\begin{itemize}
\item
If $\rho < \rho_n^-$, then with respect to $\alpha$, the sums $\sum_{j=0}^{m_n-1} P_{n,j} t^j$
and $\sum_{j=m_n}^{2m_n-1} P_{n,j} t^j$ are dominated by the summands with smallest $j$,
while $t-\lambda_{n,i}$ is dominated by $\lambda_{n,i}$.
Consequently,
$\alpha( x_n) = 1$ and
$\alpha(1-x_n) \leq (\rho/\rho_n^-)^{m_n} < 1$.
\item
If $\rho > \rho_n^+$, then with respect to $\alpha$, the sums $\sum_{j=0}^{m_n-1} P_{n,j} t^j$
and $\sum_{j=m_n}^{2m_n-1} P_{n,j} t^j$ are dominated by the summands with largest $j$,
while $t-\lambda_{n,i}$ is dominated by $t$. 
Consequently, $\alpha( 1-x_n ) = 1$ and 
$\alpha( x_n ) \leq (\rho_n^+/\rho)^{m_n} < 1$.
In particular, $\left| x_n \right|_\delta \leq (\rho_n^+/\delta)^{m_n}
< (\rho_n^+/\rho_{n+1})^{m_n} \leq 2^{-n+1}$.
\item
If $\rho \in [\rho_n^-, \rho_n^+]$, then 
\[
\alpha \left( \sum_{j=0}^{m_n-1} P_{n,j} t^j \right), 
\alpha \left( \sum_{j=m_n}^{2m_n-1} P_{n,j} t^j \right) \leq (\rho_n^+)^{2m_n-1}.
\]
For $i=1,\dots,2m_n-1$, by Lemma~\ref{L:interval} we have
\[
\alpha((t-\lambda_{n,i})^{-1}) \leq \begin{cases} \max\{\rho^{-1}, \left| \lambda_{n,i} \right|^{-1}\}
& \rho \neq \left| \lambda_{n,i} \right| \\
c^2 \left| \lambda_{n,i} \right|^{-1} & \rho = \left| \lambda_{n,i} \right|;
\end{cases}
\]
moreover, the second case can occur for at most one value of $i$. Combining, we obtain
\[
\alpha( x_n ), \alpha(1-x_n), \alpha(y_n) \leq (\rho_n^+/\rho_n^-)^{2m_n-1} c^2 \leq 16 c^2.
\]
\item
For $n>1$ and $\rho \leq \rho_{n-1}$, 
we may write $y_n - y_{n-1} = y_{n-1}(x_n - 1)$ to obtain 
\begin{align*}
\alpha( y_n - y_{n-1} ) &\leq 16c^2 \alpha( 1 - x_n ) \\
&\leq 16c^2 (\rho/\rho_n^-)^{m_n} \\
&\leq 16c^2(\rho_{n-1}/\rho_n)^{m_n} (\rho_n/\rho_n^-)^{m_n} \\
&\leq 2^{-n+5}c^2.
\end{align*}
\item
For $n>3$ and $\rho \geq \rho_{n-1}$,
we may write $y_n - y_{n-1} = y_{n-3} x_{n-1} (x_n - 1) x_{n-2}$ to obtain 
\begin{align*}
\alpha( y_n - y_{n-1}) &\leq (16c^2)^3 \alpha( x_{n-2} ) \\
&
\leq (16c^2)^3 (\rho_{n-2}^+/\rho)^{m_{n-2}} \\
&\leq (16c^2)^3 (\rho_{n-2}/\rho_{n-1})^{m_{n-2}} 
(\rho_{n-2}^+/\rho_{n-2})^{m_{n-2}} \\
&\leq 2^{-n+14} c^6.
\end{align*}
\end{itemize}
We now see that for all $n > 3$,
$\left| y_n - y_{n-1} \right|_{\spect}  \leq 2^{-n+14} c^6$.
In particular, the sequence $y_1, y_2,\dots$ is Cauchy and so has a limit $y \in A$.
By construction, we have $\left| y_n \right|_{\delta} \to 0$ as $n \to \infty$, so
$\left| y \right|_{\delta} = 0$. On the other hand, for $\rho \in [\gamma, \delta)$, 
the sequence $\left| y_n \right|_\rho$ is eventually constant, so $\left| y \right|_\rho > 0$ and in particular $y \neq 0$. In particular, $y$ is neither zero nor a unit in $A$, yielding the desired contradiction.
\end{proof}

\begin{remark} \label{R:maximal quotient}
For $A$ a Banach ring, any maximal ideal $\frakm$ of $A$ is closed \cite[Corollary~1.2.4/5]{bgr} (see also \cite[Lemma~2.2.2]{part1}), so $A/\frakm$ (topologized using the quotient norm)
is a Banach field. If $A$ is a uniform Banach algebra over a nonarchimedean field $F$ such that $\left| F^\times \right|$ is dense in $\RR^+$, 
then $A/\frakm$ is again a Banach algebra over $F$, but it need not be uniform;
consequently, Theorem~\ref{T:no Banach field not locally compact} does not  imply that $A/\frakm$ is a nonarchimedean field.
\end{remark}

\begin{remark} \label{R:remarks on discrete case}
Suppose that $A$ is a uniform Banach field which is not a nonarchimedean field,
but $F$ is discretely valued. By Lemma~\ref{L:point to field}, $\calM(A)$ contains more than one point.
For any finite extension $E$ of $F$, $E \otimes_F A$ splits as a finite direct sum, each term of which is again a Banach field which is not a nonarchimedean field. On the other hand, if $E$ is the completion of a tower of finite extensions
$F = E_0 \subseteq E_1 \subseteq \cdots$ and is not discretely valued, then 
the maps $\calM(E_{i+1} \otimes_F A) \to \calM(E_i \otimes_F A)$ are all surjective;
the uniform completion of $E \otimes_F A$ 
(i.e., the completed direct limit of the $E_n \otimes_F A$ for their spectral norms)
has spectrum equal to $\varprojlim_i \calM(E_i \otimes_F A)$, which then surjects onto $\calM(A)$ and thus also contains more than one point. By
Theorem~\ref{T:no Banach field not locally compact}, the uniform completion of $E \otimes_F A$ 
cannot be a Banach field.
One might hope that this observation can be used to extend Theorem~\ref{T:no Banach field not locally compact} to the case where $F$ is discretely valued, but we were unable to do so.
(Beware that even in this case, it is not clear that the ordinary completion of $E \otimes_F A$ is itself uniform.)

On a similar note, if $\rho \in \RR^+$ is not in the divisible closure of $\left| F^\times \right|$, then $F\{T/\rho, T^{-1}/\rho^{-1}\}$ is a nonarchimedean field which is not discretely valued, so Theorem~\ref{T:no Banach field not locally compact} implies that
$A\{T/\rho, T^{-1}/\rho^{-1}\}$ cannot be a Banach field. Note that by the latter, we mean the completion of $A[T^{\pm}]$ for the weighted Gauss norm with $\left| T \right| = \rho$, which is also the quotient of the completion of $A[T,U]$ for the weighted Gauss norm with $\left| T \right| = \rho$, $\left| U \right| = \rho^{-1}$ by the ideal $(TU-1)$. (The principal content of this statement is that the ideal is closed; see \cite[Lemma~1.5.26]{kedlaya-aws}.)

In the same context, let $S$ be the set of $\rho > 0$ which occur as $\left| T \right|_{\spect}$ for some $T \in A^\times$ for which $\left| T^{-1} \right|_{\spect} = \rho^{-1}$;
this is a group containing $\left| F^\times \right|$. By Theorem~\ref{T:no Banach field not locally compact}, $S$ cannot contain any element not in the divisible closure of $\left| F^\times \right|$.
(Namely, for any $\rho$ in the intersection arising from $T \in A$, $A$ contains the subring $F\{T/\rho,U/\rho^{-1}\}/(TU-1)$, which as per Remark~\ref{R:nondiscrete field construction} is a nonarchimedean field which is not discretely valued.)
However, this argument does not suffice to show in addition that $S$ cannot have infinite index over $\left| F^\times \right|$.
\end{remark}

\begin{remark}
Suppose that $F$ is algebraically closed.
In this case, in light of Lemma~\ref{L:point to field}, one may deduce the conclusion of 
Theorem~\ref{T:no Banach field not locally compact}  from an unpublished result of the 1973 PhD thesis of Guennebaud \cite[Proposition~IV.1]{guennebaud}:
if $A$ is a uniform Banach ring which contains a dense subfield containing $F$
and $\calM(A)$ consists of more than one point, then 
$A$ contains a zero-divisor.
(Thanks to Gabber for providing this reference.)
\end{remark}

Regarding the case where $F$ is not discretely valued, we mention the following result of Mihara \cite[Theorem~3.7]{mihara}.
\begin{theorem}[Mihara] \label{T:mihara}
For $A$ a uniform Banach field over $F$, for each $f \in A$ the ring $A\{f\} := A\{T\}/(T-f)$ is spectrally reduced; that is, the spectral norm on $A\{f\}$ is a norm. (However, it is not guaranteed that $A\{f\}$ 
is either uniform or a Banach field.)
\end{theorem}

\begin{cor}
Let $A$ be a uniform Banach field over $F$ which is a completion of $F(t)$ for some $t \in A$.
Then $A$ is \emph{sheafy} as an f-adic ring. (That is, for any ring of integral elements $A^+$ of $A$, the structure presheaf on $\Spa(A,A^+)$ is a sheaf. See \cite[Lecture~1]{kedlaya-aws} for more discussion of this condition.)
\end{cor}
\begin{proof}
Under the hypothesis on $A$, every rational localization of $A$ can be written as $A\{f\}$ for some $f \in A$.
By \cite[Proposition~2.4.20]{part1}, it then suffices to check that for every $f,g \in A$, the sequence
\begin{equation} \label{eq:mihara seq}
0 \to A \{f\} \to A\{f,g\} \oplus A\{f,g^{-1}\} \to A\{f,g^{\pm}\} \to 0
\end{equation}
is exact. (Note that \emph{a priori} we must allow $g \in A\{f\}$, but since $A$ has dense image in $A\{f\}$ we may replace $g$ with a nearby element of $A$ without changing $A\{f,g\}$ or $A\{f,g^{-1}\}$.)

By Theorem~\ref{T:mihara}, the sequence \eqref{eq:mihara seq} is exact at the left. 
Since $A\{f\}\{T^{\pm}\} \to A\{f,g^{\pm}\}$ is surjective (where $A\{f\}\{T^{\pm}\}$ is defined as in Remark~\ref{R:remarks on discrete case}), \eqref{eq:mihara seq} is exact at the right. To prove exactness at the middle, by \cite[Lemma~1.7.2]{kedlaya-aws} it suffices to check that the ideals $(T-g)A\{f\}\{T\}$ and
$(1-gT)A\{f\}\{T\}$ are closed; since the arguments are similar, we check only the first case in detail.
Note that by \cite[Corollary~1.1.14]{kedlaya-aws}, it suffices to check that the closure of $(T-g)A\{f\}\{T\}$ is finitely generated over $A\{f\}\{T\}$; it thus in turn suffices to check that $A\{f,g\}$ is a pseudocoherent module over $A\{S,T\}$ via the map taking $S$ to $f$ and $T$ to $g$. As noted above, we can rewrite $A\{f,g\}$ as $A\{h\}$ for some $h \in A$;
since $A$ is uniform, by \cite[Lemma~1.5.26]{kedlaya-aws} $A\{h\}$ is a pseudocoherent 
module over $A\{U\}$ via the map taking $U$ to $h$. Hence $A\{f,g\} = A\{h\}$ is pseudocoherent 
as a module over $A\{S,T,U\}$, and hence as a module over $A\{S,T\}$.
\end{proof}

\section{Perfectoid rings and fields}

It is not clear whether one can adapt the proof of Theorem~\ref{T:no Banach field not locally compact} to fully resolve Question~\ref{Q:Banach field} one way or the other. However, we can now answer a foundational question from the theory of perfectoid spaces,
as in \cite{part1} or \cite{scholze1}.

\begin{defn}
Fix a prime number $p$.
A \emph{perfectoid ring} is a uniform Banach ring $A$ containing a topologically nilpotent
unit $\varpi$ such that $\varpi^p$ divides $p$ in $A^\circ$ and the Frobenius map
$\varphi: A^\circ/(\varpi) \to A^{\circ}/(\varpi^p)$ is surjective.
This definition is due to Fontaine \cite{fontaine-bourbaki} and matches the one used by Kedlaya--Liu in \cite{part2}; the definitions used by Scholze in \cite{scholze1} and Kedlaya--Liu in \cite{part1} are more restrictive. 
The definition used by Bhatt--Morrow--Scholze \cite{bhatt-morrow-scholze}, modeled on that of
Gabber--Ramero \cite{gabber-ramero2}, is (slightly) more permissive.
See Remark~\ref{R:definitions of perfectoid} for further discussion.
\end{defn}

\begin{theorem} \label{T:perfectoid ring field}
Any perfectoid ring which is a Banach field is a perfectoid field.
\end{theorem}
\begin{proof}
Let $A$ be a perfectoid ring whose underlying ring is a field; then the characteristic of $A$ is either $0$ or $p$. In the latter case, $A$ is perfect, so as in Lemma~\ref{L:top nil unit}, for any $z \in A^{\circ \circ}$
we may view $A$ as a Banach algebra over the completion of $\FF_p((z))^{\perf}$ for the $z$-adic norm (for some normalization). Since the latter is a nonarchimedean field whose norm group is dense in $\RR^+$, Lemma~\ref{L:compatible norm}
and Theorem~\ref{T:no Banach field not locally compact} together imply that $A$ is a nonarchimedean field.

We may thus suppose hereafter that $A$ is of characteristic $0$; this implies that $A$ is a Banach algebra over $\QQ_p$. Apply the perfectoid (tilting) correspondence \cite[Theorem~3.6.5]{part1} to $A$ to obtain a perfectoid ring $R = A^\flat$ of characteristic $p$ with $\calM(A) \cong \calM(R)$
and a surjective homomorphism $\theta: W(R^\circ) \to A^\circ$.
By \cite[Proposition~3.6.25]{part1}, we have an identification of multiplicative monoids
\[
R \cong \varprojlim_{x \mapsto x^p} A, \qquad
r \mapsto (\dots, \theta([r^{1/p}]), \theta([r])).
\]
In particular, if $r \in R$ is nonzero, then $\theta([r^{1/p^n}]) \neq 0$ for some $n$ and hence for all $n$, and $(\dots, \theta([r^{1/p}])^{-1}, \theta([r])^{-1})$
is an element of $\varprojlim_{x \mapsto x^p} A$ corresponding to a multiplicative inverse of $r$. We conclude that $R$ is a perfectoid ring of characteristic $p$ which is a Banach field; by the previous
paragraph, $R$ is a perfectoid field. By Remark~\ref{R:single point2},
$\calM(R)$ is a single point, as then is $\calM(A)$; by Lemma~\ref{L:point to field}, $A$ is a nonarchimedean field, and hence a perfectoid field.
\end{proof}

\begin{remark} \label{R:definitions of perfectoid}
In \cite{scholze1}, the only perfectoid rings considered are algebras over perfectoid fields;
one may apply Theorem~\ref{T:no Banach field not locally compact} directly (without tilting) to show that any such ring which is a field is a perfectoid field.
However, a perfectoid ring in Fontaine's sense, or even in the sense of \cite{part1} (i.e., a Fontaine perfectoid ring which is also a $\QQ_p$-algebra) need not be a Banach algebra over any perfectoid field, so Theorem~\ref{T:no Banach field not locally compact} cannot be applied directly.

One way to see this explicitly is to construct a perfectoid ring $A$ admitting quotients isomorphic to the completions of $\QQ_p(\mu_{p^\infty})$ and $\QQ_p(p^{1/p^\infty})$.
If $A$ is a Banach algebra over some field $F$, then $F$ must be isomorphic to a subfield of the completion $F_1$ of $\QQ_p(\mu_{p^\infty})$. Note that $F_1$ must also be the completion of
$F(\mu_{p^\infty})$, whose Galois group is a closed subgroup of $\ZZ_p^\times$; by applying the Ax--Sen--Tate theorem \cite{ax} twice (to $F_1$ as a completed algebraic extension of both $\QQ_p$ and $F$), we deduce that $F$ equals either $F_1$ or $\QQ_p(\mu_{p^n})$ for some $n$. 
Similarly, $F$ must be isomorphic to a subfield of the completion $F_2$ of $\QQ_p(p^{1/p^\infty})$, and so must equal either $F_2$ or $\QQ_p(p^{1/p^n})$ for some $n$. The only choice consistent with both constraints is $F = \QQ_p$.

It remains to describe such a ring $A$ explicitly. 
Let $R_1, R_2$ be the completed perfect closures of $\FF_p \llbracket \pi_1 \rrbracket$,
$\FF_p \llbracket \pi_2 \rrbracket$;
we then have
\[
F_1 \cong W(R_1)[p^{-1}]/([\pi_1+1]^{p-1} + \cdots + [\pi_1+1] + 1),
\qquad
F_2 \cong W(R_2)[p^{-1}]/([\pi_2] - p).
\]
Let $R_3$ be the completed perfect closure of $\FF_p \llbracket \pi_1, \pi_2 \rrbracket$;
we may then take
\[
A = W(R_3)[p^{-1}]/([\pi_1+1]^{p-1} + \cdots + [\pi_1+1] + 1 - [\pi_2]).
\]
This ring is perfectoid and admits surjective morphisms $A \to F_1$, $A \to F_2$
induced by the respective substitutions $\pi_2 \mapsto 0$, $\pi_1 \mapsto 0$.
\end{remark}

\begin{remark}
Suppose that $A$ is a perfectoid ring and $\frakm$ is a maximal ideal of $A$.
As in Remark~\ref{R:maximal quotient}, $A/\frakm$ is a Banach field. 
If $A/\frakm$ is of characteristic $p$, then it is also perfect and uniform,
so Theorem~\ref{T:perfectoid ring field} implies that $A$ is a nonarchimedean field.
By contrast, if $A/\frakm$ is of characteristic $0$, then $(A/\frakm)^u$ is perfectoid
\cite[Theorem~3.6.17]{part1}, but (again as in Remark~\ref{R:maximal quotient}) this is not enough to deduce from Theorem~\ref{T:no Banach field not locally compact} that $A/\frakm$ is a nonarchimedean field.
\end{remark}

\end{document}